%% file: ABpreprint.tex
\documentclass[12pt]{amsart}
\usepackage{geometry}
\usepackage{cite}
\usepackage{enumerate}
\usepackage{amsfonts}
\usepackage{amsmath}
\usepackage{amsthm}
\usepackage{amssymb}
\usepackage{mathtools}
\usepackage{array}
\usepackage{xcolor}
\usepackage{wasysym}
\usepackage{hyperref}
\usepackage{theoremref}
\usepackage{setspace}

\newtheorem{theorem}{Main Theorem}
\newtheorem{thm}{Theorem}
\newtheorem{prop}{Proposition}
\newtheorem{lemma}{Lemma}

\title{Andrews-Bressoud Series \& Wronskians}
\author{Maggie Wieczorek}
\date{}

\begin{document}

\maketitle


\section*{Abstract}

The Andrews-Bressoud identities are one of many families of $q$-series identities relating an infinite sum to an infinite product. While the original motivation for studying these series relates to partitions, they can also be viewed in relation to irreducible characters of minimal models in the theory of vertex operator algebras. Furthermore, considering certain Wronskians of the Andrews-Bressoud series for a given modulus produces additional $q$-series, which themselves exhibit interesting and predictable modularity properties. In this paper, we connect the Andrews-Bressoud series to modular forms and prove results about the modularity of their associated Wronskians.\\

\input{intro}
\input{wk}
\input{quotient}

\section{Acknowledgements}

I would like to thank Marie Jameson for finding this project and for her guidance (and patience) throughout the process. I would also like to acknowledge and thank Shashank Kanade for sharing his vertex operator algebra expertise.

\bibliography{modularforms}
\bibliographystyle{plain}

\end{document}

%% file: intro.tex
\section{Introduction and Statement of Results}\label{section1}

Many families of $q$-series identities relate an infinite sum to an infinite product. The most famous of these are the Rogers-Ramanujan identities
\begin{align*}
\sum_{n=0}^{\infty}\frac{q^{n^2}}{(q)_n}&=\prod_{m\geq1}\frac{1}{(1-q^{5m+1})(1-q^{5m+4})}\\
\sum_{n=0}^{\infty}\frac{q^{n^2+n}}{(q)_n}&=\prod_{m\geq1}\frac{1}{(1-q^{5m+2})(1-q^{5m+3})},
\end{align*}
where $(a)_{n}\coloneqq\prod_{k=0}^{n}(1-aq^k)$, which give unexpected relationships between partition functions and have been studied and generalized by many (see Chapter 7 of \cite{andrews1976}).

One famous generalization of the Rogers-Ramanujan identities is the family of Andrews-Gordon identities, which pertain to all odd moduli including the modulus 5 of the Rogers-Ramanujan identities \cite{andrews1974}. In this paper, we discuss the Andrews-Bressoud identities
\[\sum_{n\geq0}b_{k,j}(n)q^n=\frac{\prod_{m\geq1}(1-q^{2km})(1-q^{2km-k-j+1})(1-q^{2km-k+j-1})}{\prod_{m\geq1}(1-q^m)},\]
which for $k\geq2$ and $j\in\{1,...,k\}$ generalize the Rogers-Ramanujan identities to all even moduli \cite{bressoud1979}. Here the left-hand side is the generating function for a specific partition function $b_{k,j}$ (see \cite{bressoud1979, kanadeetal2017}).

The product side of these identities resembles a generalized eta-quotient apart from a power of $q$. To this end, we define
\[a_{k,j}\coloneqq\frac{6j^2-12j-k+6}{24k}\]
and consider the functions
\[B_{k,j}(q)\coloneqq q^{a_{k,j}}\sum_{n\geq0}b_{k,j}(n)q^n=q^{a_{k,j}}\prod_{m\geq1}\frac{(1-q^{2km})(1-q^{2km-k-j+1})(1-q^{2km-k+j-1})}{(1-q^m)}.\]
Then each $B_{k,j}$ is a generalized eta-quotient \cite{robins1994, lobrich2017}. For the remainder of this paper, the term ``Andrews-Bressoud series'' refers to the functions $B_{k,j}$.

The modularity of the Rogers-Ramanujan series is well-known, and in the generalized case of the Andrews-Gordon identities, the series correspond to irreducible characters of rational vertex operator algebras (see \cite{milas2004eta, milas2004notebook, mmo2008}), which follows from the fact that the space generated by these series are invariant under the action of $\mathrm{SL}_2(\mathbb{Z})$ \cite{zhu1996}. The Andrews-Bressoud series do not correspond to such objects and are not invariant under the action of $\mathrm{SL}_2(\mathbb{Z})$; however, in Section \ref{section3} we prove that they still possess modularity properties as follows: 

\begin{thm}\label{invariance}
For a fixed $k\geq2$, the space generated by $B_{1,k}(q),...,B_{k,k}(q)$ is $\Gamma_0(2)$-invariant.
\end{thm}

We explore the modularity properties of the Andrews-Bressoud series further by taking Wronskian determinants of the series. We define the Wronskian determinant of $q$-series $f_1,...,f_k$ to be 
\[W(f_1,...,f_k)\coloneqq
\begin{vmatrix}
	f_1 & f_2 & \cdots & f_k\\
	f'_1 & f'_2 & \cdots & f'_k\\
	\vdots & \vdots & \ddots & \vdots\\
	f^{(k-1)}_1 & f^{(k-1)}_2 & \cdots & f^{(k-1)}_k
\end{vmatrix},\]
where differentiation is defined as
\[\left(\sum a(n)q^n\right)'\coloneqq\sum na(n)q^n,\]
which is the same as $\frac{1}{2\pi i}\cdot\frac{d}{d\tau}$ when $q\coloneqq e^{2\pi i\tau}$.
Furthermore, we define two Wronskian determinants specific to the Andrews-Bressoud series
\[\mathcal{W}_k(q)\coloneqq\alpha(k)\cdot
\begin{vmatrix}
	B_{1,k} & B_{2,k} & \cdots & B_{k,k}\\
	B'_{1,k} & B'_{2,k} & \cdots & B'_{k,k}\\
	\vdots & \vdots & \ddots & \vdots\\
	B^{(k-1)}_{1,k} & B^{(k-1)}_{2,k} & \cdots & B^{(k-1)}_{k,k}
\end{vmatrix}\]
and
\[\widetilde{\mathcal{W}}_k(q)\coloneqq\beta(k)\cdot
\begin{vmatrix}
	B'_{1,k} & B'_{2,k} & \cdots & B'_{k,k}\\
	B''_{1,k} & B''_{2,k} & \cdots & B''_{k,k}\\
	\vdots & \vdots & \ddots & \vdots\\
	B^{(k)}_{1,k} & B^{(k)}_{2,k} & \cdots & B^{(k)}_{k,k}
\end{vmatrix}\]
as well as the function
\[\mathcal{F}_k(\tau)\coloneqq\frac{\widetilde{\mathcal{W}}_k(q)}{\mathcal{W}_k(q)}.\]
Here, $\alpha(k)$ and $\beta(k)$ are chosen so that the $q$-expansions of $\mathcal{W}_k(q)$ and $\widetilde{\mathcal{W}}_k(q)$, respectively, have leading coefficient 1.

In \cite{milas2007weber}, Milas studies the irreducible characters of certain vertex operator algebras and gives results regarding the Wronskians of these $q$-series. In particular, he expresses the Wronskians of these characters as certain eta-quotients. The Andrews-Bressoud series are factors of these irreducible characters (see Lemma \ref{charprop}), and we can then use Milas's results to write the Wronskian $\mathcal{W}_k$ as an eta-quotient as well.

\begin{theorem}\label{wk}
The Wronskian $\mathcal{W}_k$ formed from the Andrews-Bressoud series is
\[\mathcal{W}_k(q)=\frac{\eta(\tau)^{2k^2-1}}{\eta(2\tau)^{2k-1}},\]
where $\eta(\tau)\coloneqq q^{1/24}\prod_{n\geq 1}(1-q^n)$ is the Dedekind eta function.
\end{theorem}

Then, combining this result about $\mathcal{W}_k$, facts about Wronskians and differential equations, and Theorem \ref{invariance}, we find that the modularity properties of the function $\mathcal{F}_k$ are nicer than that of the original Andrews-Bressoud series $B_{k,j}$. The following result parallels Theorem 1.2 of Milas, Mortenson, and Ono about the Andrews-Gordon series \cite{mmo2008}.

\begin{theorem}\label{wronskianquotient}
The function $\mathcal{F}_k(q)$ is a modular form of weight $2k$ for $\Gamma_0(2)$.
\end{theorem}

This paper is organized as follows: in Section \ref{section2} we connect the theory of vertex operator algebras with the Andrews-Bressoud series and provide the proof of Main Theorem \ref{wk}, and Section \ref{section3} contains results regarding the differential equation satisfied by the Andrews-Bressoud series as well as the proofs of Theorem \ref{invariance} and Main Theorem \ref{wronskianquotient}.

%% file: wk.tex
\section{Connections to vertex operator algebras and $\mathcal{W}_k$}\label{section2}

\subsection{Connections to vertex operator algebras}

The Rogers-Ramanujan series as well as its generalizations are connected to the area of vertex operator algebras (see \cite{lepowskymilne1978, lepowskywilson1981, lepowskywilson1982, berkovichmccoy1997, milas2004notebook}). The Andrews-Bressoud series are the irreducible characters of a specific vertex operator algebra (see Lemma \ref{charprop}). In order to see this, we first explore the work of Milas \cite{milas2007weber}. In Lemma 7.1 of this paper he proves that the irreducible characters of the $N=1$ superconformal minimal model in the Ramond sector can be written as (correcting a minor typographical error and setting $j\coloneqq k'+1$)
\begin{equation}\label{char7.1}
\mathrm{ch}_{k,j}(q)\coloneqq q^{\frac{(j-1)^2}{4k}}\cdot\frac{(-q)_{\infty}}{(q)_{\infty}}\sum_{n\in\mathbb{Z}}(-1)^nq^{kn^2+(j-1)n},
\end{equation}
where $j\in\{1,...,k\}$ and $(a)_{\infty}\coloneqq\prod_{n\geq0}(1-aq^n)$.

\begin{lemma}\label{charprop}
The irreducible characters in equation \eqref{char7.1} can be written as $\mathrm{ch}_{k,j}(q)=\frac{\eta(2\tau)}{\eta(\tau)}\cdot B_{k,j}(q)$ for even $k\geq2$ and $j\in\{1,...,k\}$.
\end{lemma}

\begin{proof}
First note that $\frac{(j-1)^2}{4k}=a_{k,j}+\frac{1}{24}$. Recall Jacobi's Triple Product Identity (see, for example, Theorem 2.8 of \cite{andrews1976})
\[\sum_{n\in\mathbb{Z}}z^nq^{n^2}=\prod_{n\geq0}(1-q^{2n+1})(1+zq^{2n+1})(1+z^{-1}q^{2n+1}).\]
Thus
\begin{align*}
\frac{\eta(2\tau)}{\eta(\tau)}\cdot B_{k,j}(q)&=q^{\frac{1}{24}}\cdot(-q)_{\infty}\cdot q^{a_{k,j}}\prod_{m\geq1}\frac{(1-q^{2km})(1-q^{2km-k-j+1})(1-q^{2km-k+j-1})}{(1-q^m)}\\
	&=q^{\frac{(j-1)^2}{4k}}\cdot\frac{(-q)_{\infty}}{(q)_{\infty}}\sum_{n\in\mathbb{Z}}(-q^{j-1})^n(q^k)^{n^2}\\
	&=q^{\frac{(j-1)^2}{4k}}\cdot\frac{(-q)_{\infty}}{(q)_{\infty}}\sum_{n\in\mathbb{Z}}(-1)^nq^{kn^2+(j-1)n}\\
	&=\mathrm{ch}_{k,j}(q),
\end{align*}
as desired.
\end{proof}

\subsection{Andrews-Bressoud series and differential equations}

From Theorem 5.1 (ii) in \cite{milas2007weber}, the irreducible characters of the $N=1$ superconformal minimal model in the Ramond sector discussed above form a fundamental system of solutions for the homogeneous linear $k^{\mathrm{th}}$-order differential equation
\begin{equation}\label{charode}
\left(q\frac{d}{dq}\right)^ky-\{2k(k-1)E_2(\tau)+(-k+1)E_{2,1}(\tau)\}\left(q\frac{d}{dq}\right)^{k-1}y+\cdots+F_0(\tau)y=0,
\end{equation}
where $F_n\in\mathbb{Q}[E_{2\ell},E_{2\ell,1}]$, the Eisenstein series $E_{2\ell}(\tau)$ and $E_{2\ell,1}(\tau)$ are defined via
\begin{align*}
E_{2\ell}(\tau)&\coloneqq\frac{B_{2k}}{4k}-\sum_{n\geq1}\frac{n^{2k-1}q^n}{1-q^n}\\
E_{2\ell,1}(\tau)&\coloneqq\frac{B_{2k}}{4k}+\sum_{n\geq1}\frac{n^{2k-1}q^n}{1+q^n},
\end{align*}
and $B_{2k}$ are the Bernoulli numbers. In the proof of Proposition \ref{odetranslation} we use the fact that $E_{2,1}(\tau)$ is the logarithmic derivative of $\mathfrak{f}_1(\tau)$, where $\mathfrak{f}_1$ is the Weber modular function defined by
\[\mathfrak{f}_1(\tau)\coloneqq\frac{\eta(2\tau)}{\eta(\tau)}.\]
Furthermore, it is worth noting that $E_2(\tau)$ is a quasi-modular form of weight 2 for $\mathrm{SL}_2(\mathbb{Z})$, and $E_{2,1}(\tau)$ is a modular form of weight 2 for $\Gamma_0(2)$.

Using Proposition \ref{charprop} and the technique of Lemma 6.2 in \cite{milas2004eta}, we manipulate this differential equation into one with a fundamental system given by the functions $B_{k,j}(q)$.

\begin{prop}\label{odetranslation}
After the substitution $\tilde{y}(\tau)=\frac{y(\tau)}{\mathfrak{f}_1(\tau)}$, the homogeneous differential equation \eqref{charode} becomes
\begin{equation}\label{abode}
\left(q\frac{d}{dq}\right)^k\tilde{y}-\{2k(k-1)E_2(\tau)+(-2k+1)E_{2,1}(\tau)\}\left(q\frac{d}{dq}\right)^{k-1}\tilde{y}+\cdots+P_0(\tau)\tilde{y}=0,
\end{equation}
which has a fundamental system of solutions formed by $B_{k,j}(q)$ for $j\in\{1,...,k\},$ where $P_n(\tau)$ are quasi-modular forms for $\Gamma_0(2)$.
\end{prop}

\begin{proof}
We know from Lemma \ref{charprop} that each of these irreducible characters $\mathrm{ch}_{k,j}(q)$ can be written as $\mathfrak{f}_1(\tau)\cdot B_{k,j}(q)$ and so we make the substitution $\tilde{y}(\tau)=\frac{y(\tau)}{\mathfrak{f}_1(\tau)}$ and use the differential equation \eqref{charode} as our starting point. Taking the logarithmic derivative of $\tilde{y}(\tau)$, we find
\begin{align*}
\left(q\frac{d}{dq}\right)\tilde{y}(\tau)&=\tilde{y}(\tau)\left[\frac{\left(q\frac{d}{dq}\right)y(\tau)}{y(\tau)}-E_{2,1}(\tau)\right]\\
	&=\frac{1}{\mathfrak{f}_1(\tau)}\left(q\frac{d}{dq}\right)y(\tau)-\tilde{y}(\tau)E_{2,1}(\tau),
\end{align*}
which becomes
\[\frac{1}{\mathfrak{f}_1(\tau)}\left(q\frac{d}{dq}\right)y(\tau)=\left[\left(q\frac{d}{dq}\right)+E_{2,1}(\tau)\right]\tilde{y}(\tau).\]

Then, by induction on $r$,
\[\frac{1}{\mathfrak{f}_1(\tau)}\left(q\frac{d}{dq}\right)^ry(\tau)=\left[\left(q\frac{d}{dq}\right)+E_{2,1}(\tau)\right]^r\tilde{y}(\tau).\]
If we now apply the Leibniz rule (i.e., generalized product rule), we find
\begin{equation}\label{leibnizrule}
\frac{1}{\mathfrak{f}_1(\tau)}\left(q\frac{d}{dq}\right)^ry(\tau)=\left(q\frac{d}{dq}\right)^r\tilde{y}(\tau)+rE_{2,1}(\tau)\left(q\frac{d}{dq}\right)^{r-1}\tilde{y}(\tau)+\cdots,
\end{equation}
where the dots denote terms with lower order derivatives of $\tilde{y}(\tau)$. The proof now follows after we multiply \eqref{charode} by $\frac{1}{\mathfrak{f}_1(\tau)}$ and apply \eqref{leibnizrule} for $r=1,...,k$.
\end{proof}

\subsection{Proof of Main Theorem \ref{wk}}

By Theorem 0.3 of \cite{milas2007weber}, the Wronskian formed using $\mathrm{ch}_{k,j}$ is
\[W(\mathrm{ch}_{k,1}(q),...,\mathrm{ch}_{k,k}(q))=\frac{\eta(\tau)^{2k(k-1)}}{\mathfrak{f}_1(\tau)^{k-1}}.\]

Again from Lemma \ref{charprop}, we know $\mathrm{ch}_{k,j}(q)=\frac{\eta(2\tau)}{\eta(\tau)}\cdot B_{k,j}(q)$. Then, using the well-known fact that
\[W(f\cdot f_1,...,f\cdot f_k)=f^k\cdot W(f_1,...,f_k),\]
it follows that
\begin{align*}
\frac{\eta(\tau)^{2k(k-1)}}{\mathfrak{f}_1(\tau)^{k-1}}&=W\big(\mathfrak{f}_1(\tau)\cdot B_{k,1}(q),...,\mathfrak{f}_1(\tau)\cdot B_{k,k}(q)\big)\\
	&=\mathfrak{f}_1(\tau)^k\cdot\mathcal{W}_k(q).
\end{align*}
Therefore,
\[\mathcal{W}_k(q)=\frac{\eta(\tau)^{2k^2-1}}{\eta(2\tau)^{2k-1}},\]
as desired.

%% file: quotient.tex
\section{Modularity of $\mathcal{F}_k$}\label{section3}

In order to prove Main Theorem \ref{wronskianquotient}, we first need to prove Theorem \ref{invariance}. The proof of Theorem \ref{invariance} requires that we view the Andrews-Bressoud series $B_{k,j}(q)$ as a linear combination of theta functions divided by $\eta(\tau)$. We recall a particular definition from Chapter 10 of \cite{iwaniec1997}: for $r=1$, the symmetric, positive definite matrix $A=(16k)$, the spherical function $P(m)=1$, and $N=16k$, we define the congruent theta functions for $h\in\mathbb{Z}$,
\begin{equation}\label{thetafncs}
\Theta(\tau;h)\coloneqq\sum_{m\equiv h\pmod{16k}} q^{\frac{m^2}{32k}},
\end{equation}
in order to rewrite our $B_{k,j}(q)$.

We also need the transformation properties of $\Theta(\tau;h)$ and $\eta(\tau)$ when we apply the generators of $\Gamma_0(2)$. Note that $\Gamma_0(2)$ is generated by
\begin{equation}\label{gamma02gens}T=\begin{pmatrix}1 & 1\\0 & 1\end{pmatrix}\qquad\qquad M=\begin{pmatrix}1 & 0\\ -2 & 1\end{pmatrix}=-ST^2S,\end{equation}
where $S=\begin{pmatrix} 0 & -1\\1 & 0\end{pmatrix}$. Then Propositions 10.3 and 10.4 of \cite{iwaniec1997} give the following transformation properties for our theta functions:
\begin{align}\label{thetatransformationT}
\Theta(T\tau;h)&=q^{\frac{h^2}{32k}}\cdot\Theta(\tau;h)\\
\label{thetatransformationM}\Theta(M\tau;h)&=\frac{\sqrt{1-2\tau}}{8k}\sum_{\ell\in\mathcal{H}}\left(\sum_{h'\in\mathcal{H}}e\left(\frac{h'(h'+h+\ell)}{8k}\right)\right)\Theta(\tau;\ell).
\end{align}
Then we use the transformation properties of $\eta(\tau)$ from \cite{apostol1990} to obtain
\begin{align}\label{etatransformationT}
\eta(T\tau)&=q^{\frac{1}{24}}\cdot\eta(\tau)\\
\label{etatransformationM}\eta(M\tau)&=\frac{\sqrt{1-2\tau}}{q^{\frac{1}{12}}}\cdot\eta(\tau).
\end{align}

\begin{proof}[Proof of Theorem \ref{invariance}]
First we give an alternate representation of the character $\mathrm{ch}_{k,j}$ from Proposition 2.1 of \cite{milas2007weber},
\[\mathrm{ch}_{k,j}(q)=q^{\frac{(j-1)^2}{4k}}\cdot\frac{(-q)^{\infty}}{(q)_{\infty}}\cdot\sum_{n\in\mathbb{Z}}\left(q^{2n(2kn+j-1)}-q^{(2n+1)(2kn+k-j+1)}\right).\]
Then using this, we rewrite the $B_{k,j}$ in terms of the theta functions from Equation \eqref{thetafncs} and the Dedekind eta-function:
\[B_{k,j}(q)=\frac{\Theta(\tau;2(j-1))-\Theta(\tau;4k-2(j-1))}{\eta(\tau)}.\]
Using the transformation properties for $\eta(\tau)$ and $\Theta(\tau;h)$ in Equations \eqref{thetatransformationT} - \eqref{etatransformationM}, we find that
\[B_{k,j}(Tq)=\gamma(k,j)\cdot B_{k,j}(q),\]
where $\gamma(k,j)\in\mathbb{C}$ constant, and 
\[B_{k,j}(Mq)=\sum_{j=1}^{k}\delta(k,j)\cdot B_{k,j}(q),\]
where $\delta(k,j)\in\mathbb{C}$ constant.
\end{proof}

\begin{proof}[Proof of Main Theorem \ref{wronskianquotient}]
First, well-known facts about Wronskians give that $\mathcal{F}_k(q)$ is the constant term of \eqref{abode}, which implies that $\mathcal{F}_k(q)$ is holomorphic on the upper half-plane and at the cusps of $\Gamma_0(2)$. Thus we need only show the transformation property holds.\\
Consider $A$ to be the linear transformation matrix given by the action of $T$ on the vector space spanned by $\{B_{k,j}\}$; we know this linear transformation is invertible so $\mathrm{det}(A)\neq0$. Then using Lemma 1.3 (a) of \cite{milas2010}, we have
\[\alpha(k)\cdot W(A\cdot B_{k,1}(q),...,A\cdot B_{k,k}(q))=\mathrm{det}(A)\cdot\mathcal{W}_k(q),\]
and
\[\beta(k)\cdot W(A\cdot B_{k,1}'(q),...,A\cdot B_{k,k}'(q))=\mathrm{det}(A)\cdot\widetilde{\mathcal{W}}_k(q),\]
which implies
\[\mathcal{F}_k(Tq)=\frac{\mathrm{det}(A)\cdot\widetilde{\mathcal{W}}_k(q)}{\mathrm{det}(A)\cdot\mathcal{W}_k(q)}=\mathcal{F}_k(q).\]
Now, if $C$ is the linear transformation matrix given by the action of $M$ on the vector space spanned by $\{B_{k,j}\}$ as described in Lemma \ref{invariance}, $\mathrm{det}(C)\neq0$ as the linear transformation is invertible. Again, using Lemma 1.3 (a) of \cite{milas2010}, 
\[\alpha(k)\cdot W(C\cdot B_{k,1}(q),...,C\cdot B_{k,k}(q))=\mathrm{det}(C)\cdot\mathcal{W}_k(q),\]
and
\[\beta(k)\cdot W(C\cdot B_{k,1}'(q),...,C\cdot B_{k,k}'(q))=(-2\tau+1)^{2k}\cdot\mathrm{det}(C)\cdot\widetilde{\mathcal{W}}_k(q).\]
Thus
\[\mathcal{F}_k(Mq)=\frac{(-2\tau+1)^{2k}\cdot\mathrm{det}(C)\cdot\widetilde{\mathcal{W}}_k(q)}{\mathrm{det}(C)\cdot\mathcal{W}_k(q)}=(-2\tau+1)^{2k}\cdot\mathcal{F}_k(q).\]
Therefore, $\mathcal{F}_k(q)$ is a modular form of weight $2k$ for $\Gamma_0(2)$.
\end{proof}